\documentclass[sn-mathphys,Numbered]{sn-jnl}

\usepackage{graphicx}%
\usepackage{multirow}%
\usepackage{amsmath,amssymb,amsfonts}%
\usepackage{amsthm}%
\usepackage{mathrsfs}%
\usepackage[title]{appendix}%
\usepackage{xcolor}%
\usepackage{textcomp}%
\usepackage{manyfoot}%
\usepackage{booktabs}%
\usepackage{algorithm}%
\usepackage{algorithmicx}%
\usepackage{algpseudocode}%
\usepackage{listings}%

\theoremstyle{thmstyleone}%
\newtheorem{theorem}{Theorem}
\newtheorem{proposition}[theorem]{Proposition}%

\theoremstyle{thmstyletwo}%
\newtheorem{example}{Example}%
\newtheorem{remark}{Remark}%
\newtheorem{corollary}{Corollary}%
\newtheorem{lemma}{Lemma}%

\theoremstyle{thmstylethree}%

\begin{document}

\title[Article Title]{Supremal inequalities for convex M-estimators with applications to complete and quick convergence}


\author*[1]{\fnm{Dietmar} \sur{Ferger}}\email{dietmar.ferger@tu-dresden.de}



\affil*[1]{\orgdiv{Fakult\"{a}t Mathematik}, \orgname{Technische Universit\"{a}t Dresden}, \orgaddress{\street{Zellescher Weg 12-14}, \city{} \postcode{01062} \state{Dresden}, \country{Germany}}}




\abstract{Let $\hat m_n$ be an $M$-estimator for a parameter $m$ based on a sample of size $n \in \mathbb{N}$.
We derive exponential and polynomial upper bounds for the tail-probabilities of $\sup_{k \ge n}|\hat m_k -m|$
according as a boundedness- or a moment-condition is fulfilled. This enables us to derive rates of $r$-complete convergence
and also to show $r$-qick convergence in the sense of Strasser.}




\maketitle

\section{Introduction}\label{sec1}
In statistics many parameters $m$ of interest are defined or can be characterized as a minimizer of a certain convex criterion function
$M:\mathbb{R} \rightarrow \mathbb{R}$. In mathematical optimisation or decision theory $M$ is a loss function or a cost function
depending on some (real) value and the problem is to find that value, which minimizes the pertaining loss or costs, respectively. A huge class
of such functions $M$ is given by a bivariate function $$h:S \times \mathbb{R} \rightarrow \mathbb{R}.$$
If $(S,\mathcal{S},Q)$ is a probability space, then $M$ is given by
\begin{equation} \label{M}
 M(t):= \int_S h(x,t)-h(x,t_0) Q(dx), \; t \in \mathbb{R},
\end{equation}
where $t_0$ is any fixed real number.
To ensure the existence of $M$ we assume that $h(\cdot,t):S \rightarrow \mathbb{R}$ is $\mathcal{S}$-Borel measurable for
every $t \in \mathbb{R}$ and further that $h(x,\cdot):\mathbb{R} \rightarrow \mathbb{R}$ is convex for every $x \in S$. Moreover, suppose that
\begin{equation} \label{ex}
 \int_S |D^+h(x,t)| Q(dx) < \infty  \text{ and } \int_S |D^-h(x,t)| Q(dx) < \infty\quad \forall \; t \in \mathbb{R}.
\end{equation}
Here, $D^+h(x,t)$ and $D^-h(x,t)$ denote the right derivative and the left derivative of $h(x,\cdot):\mathbb{R} \rightarrow \mathbb{R}$ at point $t \in \mathbb{R}$. Then by Proposition 2.1 of Ferger \cite{Ferger3} the criterion function $M$ exists and furthermore it is real-valued and convex.\\

\begin{remark}
Actually, one is interested in the minimizing point $m$ of $M_0$ defined by $M_0(t):=\int_S h(x,t) Q(dx)$. Since $M_0$ also has $m$ as its minimizer, both are suitable as criterion functions, although $M$ seems unnecessarily complicated and therefore less intuitive. The decisive reason for choosing $M$ is that $M_0$ only exists as a real-valued function, if and only if
\begin{equation} \label{sex}
\int_S |h(x,t)| Q(dx)< \infty \text{ for all } \quad \forall \; t \in \mathbb{R}.
\end{equation}
However, this requirement is strictly stronger than condition (\ref{ex}) as shown in Lemma 2.2 of Ferger \cite{Ferger3}. In fact, requirement (\ref{ex}) proves to be much weaker than (\ref{sex}). For instance, if $h(x,t):= |t-x|^p$ with $(x,t) \in \mathbb{R}^2$ and $p \ge 1$, then $M_0$ exists if (and only if) $\int |x|^p Q(dx) < \infty$, whereas $M$ exists if $\int |x|^{p-1} Q(dx) < \infty$. In case that $p=1$, which corresponds to the median, in fact no integrability assumption is needed.
\end{remark}

\vspace{0.3cm}
\begin{remark} \label{pM}
In Ferger \cite{Ferger3} we derive several analytical properties of the criterion function $M$ with emphasis on the existence, characterisation and
possible uniqueness of its minimizers, see section 3 below for a summary.
\end{remark}

\vspace{0.3cm}
Usually $Q$ is unknown, but it can be approximated by the empirical measure $Q_n$ pertaining to a sample $X_1,\ldots,X_n, n \in \mathbb{N},$ of  $S-$valued random
variables defined on a common probability space $(\Omega,\mathcal{A},\mathbb{P})$, that is $Q_n := n^{-1} \sum_{i=1}^n \delta_{X_i}$, where
$\delta_x$ denotes the Dirac-measure at point $x \in S$. Replacing $Q$ in the definition (\ref{M}) of $M$ by $Q_n$ leads to
the empirical counterpart $M_n$ given by
\begin{equation} \label{Mn}
 M_n(t) = \frac{1}{n} \sum_{i=1}^n h(X_i,t)-h(X_i,t_0), \; t \in \mathbb{R}.
\end{equation}
Let $X_1,\ldots,X_n$ be independent with common distribution $Q$ and assume (only for that moment) that (\ref{sex}) holds. Then by the Strong Law of Large Numbers $M_n(t) \rightarrow M(t) \; \mathbb{P}-$almost surely (a.s.) for each $t \in \mathbb{R}$. Let
\begin{equation*}
\hat{m}_n \text{ be the \textbf{smallest minimizing point} of } M_n.
\end{equation*}
Of course $\hat{m}_n$ is also the (smallest) minimiser of $t \mapsto 1/n \sum_{1 \le i \le n} h(X_i,t)$. By Corollary 6 of Ferger \cite{Ferger0} the estimator $\hat{m}_n$ for $m$ is
$\mathcal{A}$-Borel measurable. If $M$ is coercive and $m$ is unique, then $\hat{m}_n$ converges to $m$ a.s. by Theorem 9 of Ferger \cite{Ferger0}. Now it is well known that almost sure convergence is equivalent to
$\lim_{n \rightarrow \infty}\mathbb{P}(\sup_{k \ge n}|\hat m_k -m|>x)=0$ for all $x>0$. We will improve this result significantly by showing that here
the tail probabilities of the suprema converge exponentially fast or polynomially fast, depending on whether a boundedness condition or a moment condition is satisfied.\\

The random variable $\hat{m}_n$ is called \emph{convex M-estimator}, because the pertaining criterion function $M_n$ is convex. A first systematic study of such convex M-estimators based on $M_n$ of general type (\ref{Mn}) was done by Habermann (1989) \cite{Habermann} and Niemiro (1992) \cite{Niemiro}. Habermann \cite{Habermann} focuses on consistency and $\sqrt{n}-$ asymptotic normality, whereas Niemiro \cite{Niemiro} derives an asymptotic linearization of $\sqrt{n}(\hat{m}_n-m)$  with almost sure rates of convergence for the remainder term. Moreover, he extends the consistency results by giving rates of convergence for the tail-probabilities, confer (\ref{Niemiro}) and Proposition \ref{Thm2} below for details. It should be noted that both authors more generally consider $M_n(t)$ with
parameter $t$ running through the euclidian space $\mathbb{R}^d$ with $d \ge 1$ and not only $d=1$ as we do. An alternative approach to the asymptotic linearization is given by Hjort and Pollard (1993) \cite{Hjort}. Here, the key idea is to represent $\sqrt{n}(\hat{m}_n-m)$ as a minimizer of the convex process $Z_n$, which arises from $M_n$ by  $Z_n(t):= n \{M_n(m+n^{-1/2}t)-M_n(m)\}, t \in \mathbb{R}^d$. Under some conditions on $h$ and $M$ one has that $Z_n$ allows for a decomposition into a linear part plus a quadratic form and a negligible additive term. Hjort and Pollard  prove an argmin-theorem for such convex stochastic processes and herewith obtain asymptotic normality, confer also Liese and Mieschke (2008) \cite{Liese}.
Davis, Knight and Liu (1992) give an extension to general convex processes $Z_n$. In short their argmin-theorem says that if the fidis of the $Z_n$ converge to $Z$
and if $Z$ has a.s. an unique minimizer, $\xi$ say, then the argmins of $Z_n$, $\xi_n$ say, converge in distribution to $\xi$. In a further step of generalization Geyer (1996) \cite{Geyer} considers $Z_n$ and $Z$ with values in the extended real line $\overline{\mathbb{R}}=[-\infty,\infty]$. Geyer's result brings forth non-normal limits and normalizing sequences different from $\sqrt{n}$. Here the uniqueness of $\xi$, the argmin of $Z$, is essential, because
$\xi$ acts as limit variable. If uniqueness is not given, Ferger (2021) \cite{Ferger3} shows that the $\xi_n$ converge in distribution to the \textbf{set} of all minimizer $Z$. To explain this notion of ''distributional convergence to a set'' first recall that classical distributional convergence
means that the distributions converge weakly to a probability measure, namely the distribution of the limit variable. In case of convergence to a set
the limit is no longer a probability measure, but a \emph{Choquet-capacity}, which in turn uniquely determines a \emph{random closed set}. (For the notions of \emph{Choquet-capacity} and \emph{random closed set} we refer to the monograph of Molchanov \cite{Molchanov}.) The limit of the $\xi_n$ is the capacity functional of the random closed set of all minimizers of $Z$. For unique $\xi$ this capacity functional coincides with the distribution of $\xi$ and we are back to the usual convergence in distribution.\\

In contrast to the above asymptotics (mainly for the distributions) the novelty in this work is that we give upper bounds for the tail probabilities $\mathbb{P}(\sup_{k \ge n}|\hat m_k -m|>x), x>0,$ that hold for every finite sample size $n \in \mathbb{N}$. From these inequalities, we derive $r$-complete convergence (including rates) and
$r$-quick convergence.  This is completely new in convex M-estimation.
The concept of $r-$complete convergence was for $r=1$ as first introduced by Hsu and Robbins (1947) \cite{Hsu} and later on generalized to $r>0$ by
Tartakovsy (1998) \cite{Tarta1}, confer also Tartakovsky (2023) \cite{Tarta2}. Whereas the mode of $r$-quick convergence goes back to Strassen (1967) \cite{Strassen}.\\

The paper is organized as follows: In section 2 we derive exponential inequalities for the suprema
$\sup_{k \ge n}|\hat m_k -m|$.
Herewith a result of Niemiro \cite{Niemiro} can be improved substantially (at least for dimension $d=1$). Section 3 is devoted to the analysis of the
criterion function $M$. The important special case of quantile-estimators is treated in section 4 with an improvement of Serfling's \cite{Serfling} inequality. Section 5 contains the above mentioned results on $r$-complete and $r$-quick convergence, but also many examples. One of these examples
gives an extension of Hoeffdings's inequality. Finally, in section 6 we derive the corresponding propositions when instead of boundedness only a moment-conditions is fulfilled.

\section{Exponential tail-bounds}
In order to derive exponential inequalities for the tail-probabilities of the suprema we need that the following boundedness-assumption is fulfilled:
For every $x \neq 0$ there exist reals $a(x)<b(x)$ such that
\begin{equation} \label{b}
 a(x)\le D^+h(X_i,m+x)\le b(x) \text{ a.s. for all }  1\le i \le n.
\end{equation}

\vspace{0.3cm}
Let $D^+f$ and $D^-f$ denote the right-and left derivative of a function $f:\mathbb{R} \rightarrow \mathbb{R}$.
It is well known that $M$ as a convex function possesses the right and left derivatives, which in addition are non-decreasing. Moreover, $m \in \mathbb{R}$ is a minimizer of $M$ if and only if
\begin{equation} \label{m}
 D^-M(m) \le 0 \le D^+M(m).
\end{equation}
Confer Corollary 7.2 of Ferger \cite{Ferger3}, where we give a short proof.

Our first main result is as follows:\\

\begin{theorem} \label{expb} Suppose $m \in \mathbb{R}$ satisfies (\ref{m}). Further assume that $X_1,\ldots,X_n, n \in \mathbb{N}$, are i.i.d. with distribution $Q$ such that (\ref{b}) is fulfilled. Then
the following exponential inequalities hold for every $n \in \mathbb{N}$:
\begin{align}
&\mathbb{P}(\sup_{k \ge n} (\hat{m}_k-m) >x) \le \exp\{-2(b(x)-a(x))^{-2} n D^+M(m+x)^2\} \quad \forall \; x>0. \label{i1}\\
&\mathbb{P}(\inf_{k \ge n} (\hat{m}_k-m) < -x) \le \exp\{-2(b(-x)-a(-x))^{-2} n D^+M(m-x)^2\} \quad \forall \; x>0.\label{i2}\\
&\mathbb{P}(\sup_{k \ge n} |\hat{m}_k-m| >x) \le 2 \exp\{-2 A(x)^{-2} n d(x)^2\} \quad \forall \; x>0, \label{i3}
\end{align}
where $$A(x):=\max\{b(x)-a(x),b(-x)-a(-x)\}$$ and  $$d(x):=\min\{D^+M(m+x), -D^+(m-x)\} \ge 0.$$ If $m$ is the unique minimizing point of $M$, then actually $d(x)$ is positive for
every $x>0$ and in (\ref{i1})-(\ref{i3}) we obtain in each case an exponentially fast rate of convergence.
\end{theorem}

\begin{proof} For the derivation of (\ref{i1}) first observe that
$$
\{\sup_{k \ge n} (\hat{m}_k-m) >x\}= \bigcup_{k \ge n}\{\hat{m_k}-m>x\}=\bigcup_{l \ge n}\bigcup_{n \le k \le l}\{\hat{m_k}-m>x\},
$$
whence by $\sigma$-continuity from below
\begin{align} \label{obs1}
\mathbb{P}(\sup_{k \ge n} (\hat m_k-m)>x)=\lim_{l \rightarrow \infty}\mathbb{P}(\bigcup_{n \le k \le l}\{\hat m_k-m>x\}).
\end{align}

By Theorem 1 of Ferger \cite{Ferger0} we have the following basic equality: $\{\hat{m}_k>m+x\}=\{D^+M_k(m+x) <0\}$. Herewith it follows that
\begin{eqnarray}
\{\hat{m}_k-m>x\}&=&\{D^+M_k(m+x)<0\} \nonumber\\
&=&\{\frac{1}{k}\sum_{i=1}^k D^+h(X_i,m+x)<0\} \nonumber\\
&=&\{\frac{1}{k}\sum_{i=1}^k \big(D^+h(X_i,m+x)-\mathbb{E}[D^+h(X_i,m+x)]\big)<-\mathbb{E}[D^+h(X_i,m+x)]\}\nonumber\\
&=&\{\frac{1}{k}\sum_{i=1}^k -\big(D^+h(X_i,m+x)-\mathbb{E}[D^+h(X_i,m+x)]\big)>\mathbb{E}[D^+h(X_i,m+x)]\} \nonumber\\
&=&\{\frac{1}{k}\sum_{i=1}^k -\big(D^+h(X_i,m+x)-\mathbb{E}[D^+h(X_i,m+x)]\big)>D^+M(m+x)\}, \label{rep}
\end{eqnarray}
where the last equality holds, because $\mathbb{E}[D^+h(X_i,t)]=D^+M(t)$ for all $t \in \mathbb{R}$ according to Proposition 2.4 of Ferger \cite{Ferger3} upon noticing that the $D^+h(X_i,t)$ are integrable by (\ref{b}). Put $\xi_i:=\xi_i(x,m):=-\big(D^+h(X_i,m+x)-\mathbb{E}[D^+h(X_i,m+x)]\big), i \in \mathbb{N}$. Then $(\xi_i)_{i \in \mathbb{N}}$ is a sequence
of i.i.d. integrable real random variable. If
$$
 T_k := \frac{1}{k} \sum_{i=1}^k \xi_i \quad \text{and} \quad \mathcal{G}_k :=\sigma(T_j:j \ge k), \; k \in \mathbb{N},
$$
then by Example 6.5.5 (c) in G\"{a}nssler and Stute \cite{Stute} the sequence $(T_k,\mathcal{G}_k)_{k \in \mathbb{N}}$ is a reverse martingale, confer also Example 3 on p. 241 in Chow and Teicher \cite{Chow}.
Infer from the monotonicity of $D^+M$ and (\ref{m}) that $\eta:=D^+M(m+x) \ge D^+M(m) \ge 0$ is non-negative. Thus with (\ref{rep}) we obtain for every $l \ge n$:
\begin{eqnarray}
 \bigcup_{n \le k \le l} \{\hat{m}_k-m >x\}&=&\bigcup_{n \le k \le l}\{T_k > \eta\}=\bigcup_{1 \le j \le l-n+1}\{T_{l-j+1}> \eta\} \label{mart}\\
 &=&\bigcup_{1 \le j \le l-n+1}\{ e^{h T_{l-j+1}}> e^{h \eta}\} \nonumber\\
 &=&\{\max_{1 \le j \le l-n+1} e^{h T_{l-j+1}}> e^{h \eta}\} \quad \forall \; h>0. \label{Chernov}
\end{eqnarray}
The reversing of time in (\ref{mart}) leads to the martingale $(T_{l-j+1},\mathcal{G}_{l-j+1})_{1 \le j \le l-n+1}$. Since $x \mapsto e^{h x}, x \in \mathbb{R}$, is convex for each $h>0$, it follows that $(e^{h T_{l-j+1}},\mathcal{G}_{l-j+1})_{1 \le j \le l-n+1}$ is a non-negative sub-martingale. Hence
Doob's maximal-inequality yields that
$$
 \mathbb{P}\big(\max_{1 \le j \le l-n+1} e^{h T_{l-j+1}}> e^{h \eta}\big) \le e^{-h \eta} \mathbb{E}[e^{h T_n}] \quad \forall \; h>0.
$$
Notice that the upper bound here does not depend on $l \ge n$. Therefore by (\ref{obs1}) and (\ref{Chernov}) we arrive at
\begin{equation} \label{ub1}
 \mathbb{P}(\sup_{k \ge n} (\hat{m_k}-m)>x) \le e^{-h \eta} \mathbb{E}[e^{h T_n}] \quad \forall \; h>0.
\end{equation}
If $S_n:=\sum_{i=1}^n \xi_i$, then $\mathbb{E}[e^{h T_n}] = \mathbb{E}[e^{u S_n}]$, where $u:=h/n >0$. By definition
$\xi_i= \rho_i-\mathbb{E}[\rho_i]$, where $\rho_i = -D^+h(X_i,m+x), i \in \mathbb{N}$, whence $-b(x) \le \rho_i \le -a(x)$ a.s. for all $1 \le i \le n$ by the boundedness assumption (\ref{b}). Now it follows from Hoeffding's lemma, confer, e.g., Lemma 2.6 in Massart (2007) \cite{Massart} that
$$
 \mathbb{E}[e^{u S_n}] \le e^{\frac{1}{8} n u^2 (a(x)-b(x))^2} \quad \forall \; u>0.
$$
Consequently we obtain with (\ref{ub1}) that
\begin{equation} \label{ub2}
 \mathbb{P}(\sup_{k \ge n} (\hat{m}_k-m)>x) \le e^{-h \eta + \frac{1}{8}\frac{h^2}{n} (b(x)-a(x))^2} \quad \forall \; h>0.
\end{equation}
An elementary minimization of the upper bound with respect to $h$ results in $h_0=\frac{4n \eta}{(b(x)-a(x))^2}$ as
minimizer, which is positive, if $\eta$ is positive. Substituting $h_0$ into the upper bound on the right-hand side of (\ref{ub2}) gives
$$
 \mathbb{P}(\sup_{k \ge n} (\hat{m_k}-m)>x) \le e^{-\frac{2}{(b(x)-a(x))^2} n \eta^2},
$$
which is inequality (\ref{i1}). In case that $\eta=0$ taking the limit $h \downarrow 0$ in (\ref{ub2}) yields $1$ as (trivial) upper bound in accordance with (\ref{i1}). As to the derivation of inequality (\ref{i2}) one analogously starts with
\begin{equation} \label{inf1}
\mathbb{P}(\inf_{k \ge n} (\hat{m_k}-m)<-x)=\lim_{l \rightarrow \infty}\mathbb{P}(\bigcup_{n \le k \le l}\{\hat{m_k}-m<-x\}).
\end{equation}
Now it is true that
\begin{eqnarray}
& &\{\hat{m}_k-m <-x\} \subseteq \{\hat{m}_k-m \le -x\} = \{\hat{m}_k \le m-x\} \nonumber\\
&=& \{D^+M_k(m-x) \ge 0\} \label{sigma}\\
&=&  \{\frac{1}{k}\sum_{i=1}^k D^+h(X_i,m-x)-\mathbb{E}[D^+h(X_i,m-x)]\big)\ge -D^+M(m-x)\}, \label{inf2}
\end{eqnarray}
where the equality in (\ref{sigma}) holds by Theorem 1 of Ferger \cite{Ferger0}.
Put $\eta:=-D^+M(m-x)$. Since $m-x<m$, it follows from Theorem 24.1 in Rockafellar that $D^+M(m-x) \le D^-M(m)$. But $D^-M(m)\le 0$ by (\ref{m}), whence $\eta \ge 0$. From here on, the rest of the proof of (\ref{i2}) is completely analogous to the one above.

Finally, the third inequality (\ref{i3}) follows from what we have shown so far, because
\begin{equation} \label{supinf}
 \{\sup_{k \ge n}|\hat{m}_k-m|>x\} = \{\sup_{k \ge n}(\hat{m}_k-m)>x\} \cup \{\inf_{k \ge n}(\hat{m_k}-m)<-x\}.
\end{equation}
It remains to prove the last part of the theorem. For that purpose assume that there exists some positive $x$ such that
$d(x)=0$, so that $D^+M(m+x)=0$ or $D^+M(m-x)=0$. In the first case it follows by monotonicity that $D^+M(z)=0$ for all $z \in [m,m+x]$.
Since the set of points at which $D^+M$ and $D^-M$ coincide lies dense in $\mathbb{R}$, we find a point $z \in (m,m+x)$ with
$D^-M(z)=D^+M(z)=0$. Consequently $z$ is a further minimizer of $M$  by (\ref{m}), in contradiction to the uniqueness of $m$. In the second case
notice that $0=D^+M(m-x)\le D^-M(m-x/2) \le D^-M(m)\le 0$, where the first inequality holds by Theorem 24.1 of Rockafellar \cite{Rockafellar}.
Thus $D^-M(z)=0$ for all $z \in [m-x/2,m]$ and one can argue as in the first case to arrive at a contradiction to the uniqueness of $m$. Herewith the proof is complete.
\end{proof}

Assume that $m$ is unique. Niemiro in Theorem 3 \cite{Niemiro} shows  that for every $x>0$ there exist constants $\alpha(x)>0$ and $K(x)$ such that
\begin{equation} \label{Niemiro}
\mathbb{P}(|\hat{m}_n-m|>x)  \le K(x) \exp\{-\alpha(x) n\} \quad \forall \; x>0 \quad \forall \; n \in \mathbb{N}.
\end{equation}
Comparing this with inequality (\ref{i3}) one sees that
the constants can be identified: $K(x)=2$ and $\alpha(x)= 2 \{d(x)/A(x)\}^2 >0$. But actually more remarkable is that our inequality even applies to the larger tail-probabilities  for the suprema, $\mathbb{P}(\sup_{k \ge n}|\hat{m}_k-m|>x)\ge \mathbb{P}(|\hat{m}_n-m|>x)$, namely:
\begin{equation} \label{supN}
 \mathbb{P}(\sup_{k \ge n}|\hat{m}_k-m|>x) \le 2 \exp\{-\alpha(x) n\} \quad \forall \; x>0 \quad \forall \; n \in \mathbb{N}.
\end{equation}
This is a significant improvement of Niemiro's result (\ref{Niemiro}). On the other hand Niemiro in comparison to our boundedness-assumption (\ref{b}) merely requires the existence of a moment generating
function. More precisely, he requires there exists some $x_0>0$ such that for every $x \in [-x_0,x_0]$ there exists some $\theta_0=\theta_0(x)>0$ so that
\begin{equation} \label{moment}
 \mathbb{E}[ \exp\{\theta D^+M(X_1,m+x)\}] < \infty \quad \forall \; \theta \in (0,\theta_0).
\end{equation}
We will see that a result of type (\ref{supN}) still holds under the weaker requirement (\ref{moment}).

To do this, note that that under (\ref{moment}) all moments of the random variables $D^+M(X_1,m+x), |x| \le x_0$, exist and are finite, confer, e.g., Billingsley \cite{Billingsley}, p.285.
In particular, the $\xi_i=\xi_i(x)$ in the above proof of Theorem \ref{expb} are integrable. Therefore we may apply (\ref{ub1}) with $h=n \theta$ and $0<\theta<\theta_0$. It gives:

\begin{equation}
 \mathbb{P}(\sup_{k \ge n} (\hat{m_k}-m)>x) \le e^{-n \theta \eta} \mathbb{E}[e^{\theta S_n}]= e^{-n \theta \eta} \varphi_x(\theta)^n,
\end{equation}
where $\varphi_x(\theta)=\mathbb{E}[\exp\{\theta \xi_1(x)\}]$ is the common moment generating function of the $\xi_i(x)$. Letting $\kappa_x(\theta):=\log \varphi_x(\theta)$ we arrive at
$$
 \mathbb{P}(\sup_{k \ge n} (\hat{m_k}-m)>x) \le \exp\{-n(\eta \theta-\kappa_x(\theta)\} \quad \forall \; x \in (0,x_0].
$$
Recall from Theorem \ref{expb} that $\eta=\eta_x=D^+M(m+x)$ is positive as long as $x$ is positive. Thus Lemma 2.7.2 of Durrett \cite{Durrett} yields that
the factor $\eta \theta - \kappa_x(\theta) =:\beta_1(x)$ is positive provided $\theta$ is sufficiently small, say $\theta \le \theta_1(x)$ for some positive $\theta_1(x)$. In summary we get:
\begin{equation} \label{e1}
 \mathbb{P}(\sup_{k \ge n} (\hat{m_k}-m)>x) \le \exp\{-\beta_1(x) n\} \quad \forall \; x \in (0,x_0]  \quad \forall \; n \in \mathbb{N}.
\end{equation}
Analogously it follows that there exists some positive $\beta_2(x)$ such that
\begin{equation} \label{e2}
 \mathbb{P}(\inf_{k \ge n} (\hat{m_k}-m)<-x) \le \exp\{-\beta_2(x) n\} \quad \forall \; x \in (0,x_0]  \quad \forall \; n \in \mathbb{N}.
\end{equation}
Combining (\ref{supinf}) with the above inequalities (\ref{e1}) and (\ref{e2}) finally results in
\begin{equation} \label{e3}
 \mathbb{P}(\sup_{k \ge n} |\hat{m_k}-m|>x) \le 2 \exp\{-\beta(x) n\} \quad \forall \; x \in (0,x_0]  \quad \forall \; n \in \mathbb{N},
\end{equation}
where $\beta(x)= \min\{\beta_1(x),\beta_2(x)$ is positive. If we replace $\beta(x)$ by $\beta(x\wedge x_0)$, then the inequality holds actually for all $x>0$. Thus we have indeed shown the counterpart of inequality (\ref{supN}), but under the assumption (\ref{moment}), which is weaker than the boundedness-assumption (\ref{b}) used in the derivation of (\ref{supN}).

\section{Existence and uniqueness of $m$}
In this section, we summarise results from Ferger \cite{Ferger3} about analytical properties of M and its \emph{minimum set} $A(M)$, in terms of its existence, characterisation and uniqueness.
Here, $A(f):=\{t \in \mathbb{R}: f(t)=\inf_{s \in \mathbb{R}} f(s)\}$ is the set of all minimizing arguments of an arbitrary function $f:\mathbb{R} \rightarrow \mathbb{R}$.\\

If condition (\ref{ex}) holds, then for each fixed real number $t_0$ the pertaining criterion function $M$ in (\ref{M})
exists and is real-valued for each $t \in \mathbb{R}$. Moreover $M:\mathbb{R} \rightarrow \mathbb{R}$ is convex and
\begin{equation} \label{DM}
D^\pm M(t)= \int_S D^\pm h(x,t) Q(dx) \text{ for every } t \in \mathbb{R}.
\end{equation}
For each distribution $Q$ the minimum set is given by
\begin{eqnarray}
 A(M)&=&\{m \in \mathbb{R}:D^-M(m) \le 0 \le D^+M(m)\} \label{AM}\\
 &=&\{m \in \mathbb{R}: \int_S D^- h(x,t) Q(dx) \le 0 \le \int_S D^+ h(x,t) Q(dx)\}.
\end{eqnarray}
In particular, $Q_n$ instead of $Q$ yields $A(M_n)$.
In the further course we consider the rich subclass of estimators induced by bivariate functions of the type $h:\mathbb{R}\times \mathbb{R} \rightarrow \mathbb{R}$ given through
\begin{equation} \label{h}
h(x,t)=v(t-x) \text{ where } v:\mathbb{R} \rightarrow \mathbb{R} \text{ is convex}.
\end{equation}
Furthermore, for the sake of simplicity, we consider only $t_0=0$.
So let $$M(t)=\int_\mathbb{R} v(t-x)-v(-x) Q(dx)$$ with general probability measure $Q$. Notice that (\ref{DM}) reduces to
\begin{equation} \label{Dv}
D^\pm M(t)= \int_S D^\pm v(t-x) Q(dx) \text{ for every } t \in \mathbb{R}.
\end{equation}
If $v$ is not only convex but coercive with $D^-v(0)\le 0 \le D^+v(0)$, (which means that $0$ is a minimizer of $v$), then $M$ is coercive and Argmin$(M)$ is a non-empty compact interval. In particular, $m \in$ Argmin$(M)$ exists. For $Q$ symmetric at $m \in \mathbb{R}$, i.e., $X_1-m$ and $m-X_1$ are equal in distribution, one has that $m \in A(M) \neq \emptyset$ provided $v$ is convex and even.
In both cases, if in addition $v$ is differentiable on $\mathbb{R}\setminus \{0\}$ and $F$ is continuous on Argmin$(M)$, then Argmin$(M)$ is non-empty and given by
\begin{eqnarray}
\text{Argmin}(M) &=& \{m \in \mathbb{R}: \int_\mathbb{R} D^- v(m-x) Q(dx)=0= \int_\mathbb{R} D^+ v(m-x) Q(dx)\} \label{sol0}\\
                                &=& \{m \in \mathbb{R}: \int_{\mathbb{R} \setminus \{m\}} v^\prime (m-x) Q(dx)=0\}. \label{sol}
\end{eqnarray}
Finally, Argmin$(M)=\{m\}$ is a singleton in case that $v$ is strictly convex.
For $v$ differentiable on the entire real line the continuity assumption on $F$ can be dropped and the integral in (\ref{sol}) is over the entire space $\mathbb{R}$ rather than over $\mathbb{R} \setminus \{m\}$.

In summary, it can be said that under strict convexity (and the above requirements), a unique minimizing point exists and this results as the solution of the integral equation $\int_\mathbb{R} v^\prime (m-x) Q(dx)=0$ or
$\int_{\mathbb{R} \setminus \{m\}} v^\prime (m-x) Q(dx)=0$, according as $v$ is differentiable at zero or not.

\section{Quantile estimators}

Before we come to the conclusions of our theorem, we would first like to illustrate it with an important example.
For a given $\alpha \in (0,1)$ define $v(u):=v_\alpha(u):=u(1_{\{u \ge 0\}}-\alpha)$.
The one-sided derivatives are given by $D^+v(u)=1_{\{u \ge 0\}}-\alpha$ and $D^-v(u)=1_{\{u > 0\}}-\alpha$, so that by (\ref{Dv}) one obtains $D^+M(t)=F(t)-\alpha$ and $D^-M(t)=F(t-)-\alpha$, where $F$ is the distribution function of $Q$ and $F(t-)$ is the left-limit of $F$ at point $t \in \mathbb{R}$. It follows from (\ref{AM}) that $A(M)$ is equal to the set of all $\alpha$-quantiles $q_\alpha$ of $F$. Similarly, if
$F_n$ is the empirical distribution function of the sample $X_1,\ldots,X_n$, then $\hat m_n=F^{-1}_n(\alpha)$ with $F^{-1}$ denoting the quantile function (general inverse) of any distribution function $F$. The boundedness assumption is fulfilled with $b(x)=1-\alpha$ and $a(x)=-\alpha$, so that
$b(x)-a(x)=1$. Moreover $d(x)=\min\{F(q_\alpha+x)-\alpha,\alpha-F(q_\alpha-x)\}$. If $q_\alpha$ is unique we obtain from Theorem \ref{expb} that $d(x)$ is positive and
\begin{equation} \label{q}
 \mathbb{P}(\sup_{k \ge n} |F^{-1}_k(\alpha)-F^{-1}(\alpha)| >x) \le 2 e^{-2 n d(x)^2} \quad \forall \; x>0.
\end{equation}
Serfling \cite{Serfling} shows in Theorem 2.3.2 that
\begin{equation} \label{S}
 \mathbb{P}(|F^{-1}_n(\alpha)-F^{-1}(\alpha)| >x) \le 2 e^{-2 n d(x)^2} \quad \forall \; x>0.
\end{equation}
In view of $\sup_{k \ge n} |F^{-1}_k(\alpha)-F^{-1}(\alpha)| \ge |F^{-1}_n(\alpha)-F^{-1}(\alpha)|$ our inequality (\ref{q}) is significantly sharper
than Serfling's inequality (\ref{S}) with the same bound. In Corollary 2.3.2 Serfling \cite{Serfling} gives the following supremal-inequality:
\begin{equation}
\mathbb{P}(\sup_{k \ge n} |F^{-1}_k(\alpha)-F^{-1}(\alpha)| >x) \le \frac{2}{1-\rho(x)} e^{-2 n d(x)^2} \quad \forall \; x>0,
\end{equation}
where $\rho(x)=e^{-2 d(x)^2}$. We see that not only $2 < \frac{2}{1-\rho(x)}$ for all $x>0$, but also that $\frac{2}{1-\rho(x)} \uparrow \infty$ as $x\downarrow 0$. The reason for this lies in that Serfling argues as follows:
\begin{eqnarray}
\mathbb{P}(\sup_{k \ge n} |F^{-1}_k(\alpha)-F^{-1}(\alpha)| >x)&=&\mathbb{P}(\bigcup_{k \ge n} \{|F^{-1}_k(\alpha)-F^{-1}(\alpha)|>x\})\nonumber \\
 &\le& \sum_{k \ge n} \mathbb{P}(|F^{-1}_k(\alpha)-F^{-1}(\alpha)| >x) \label{subadd}\\
&\le& 2 \sum_{k \ge n} (e^{-2 d(x)^2})^k \quad \text{by } (\ref{S}) \nonumber
\end{eqnarray}
and the geometric sum does the rest. Obviously the upper bound in (\ref{subadd}) based on the $\sigma$-subadditivity of $\mathbb{P}$ is far from doing the same thing as Doob's maximal-inequality for martingale does in our proof.

\section{Conclusions in the general framework}
The boundedness requirement (\ref{b}) is intended to apply to the entire section, without further constant mention of it.

We can conclude from Theorem \ref{expb} that the suprema $\sup_{k \ge n} |\hat{m}_k-m|, n \in \mathbb{N},$ of the distances
converge completely to $0$, provided $m$ is unique, that is
\begin{equation} \label{cc}
 \sum_{n=1}^\infty \mathbb{P}(\sup_{k \ge n} |\hat{m}_k-m|>x) < \infty \quad \text{for all } x>0.
\end{equation}
This implies that $\sup_{k \ge n} |\hat{m}_k-m| \rightarrow 0$ a.s. from which in turn we  can conclude the a.s. convergence of $\hat m_n$ to $m$.

In fact one can upgrade the complete convergence (\ref{cc}) substantially under the assumption that there exist positive constants $c$ and $\delta$
such that

\begin{equation} \label{DpMlb}
 |D^+M(m+x)| \ge c \; |x| \quad \forall \; x \in [-\delta,\delta].
\end{equation}

In Lemma \ref{pivot} below we give a very huge class of examples, where the above inequality (\ref{DpMlb}) is fulfilled.

\vspace{0.4cm}
\begin{corollary} \label{rcomplconv} Suppose that in (\ref{b}) the one-sided limits of the lower and upper bounds $a$ and $b$ at $0$ exist and are finite. If $D^+M$ meets the requirement (\ref{DpMlb}), then
$$
 n^\gamma \; \sup_{k \ge n}|\hat{m}_k-m|
$$
\emph{converges $r$-completely }to $0$ for all $r>0$, whenever $\gamma < 1/2$, that is by definition
\begin{equation}
 \sum_{n=1}^\infty n^{r-1} \mathbb{P}\big(n^\gamma \sup_{k \ge n}|\hat{m}_k-m| > \epsilon\big) <\infty \quad \forall \; \epsilon>0.
\end{equation}
We write
\begin{equation} \label{rcompl}
\sup_{k \ge n}|\hat{m}_k-m|=o(n^{-\gamma}) \; \mathbb{P} \text{-r completely}.
\end{equation}
\end{corollary}

\begin{proof} If $x \in (0,\delta]$, then $D^+M(m+x) \ge D^+M(m) \ge 0$, whence $D^+M(m+x)=|D^+M(m+x)| \ge c |x| = c x$.
If $x \in [-\delta,0)$, then $D^+M(m+x) \le D^-M(m) \le 0$, where the first inequality holds by Theorem 24.1 of Rockafellar \cite{Rockafellar}.
Consequently, $-D^+M(m+x)=|D^+M(m+x)| \ge c |x| = c (-x)$. Replacing $x$ by $-x$ yields that $-D^+M(m-x) \ge c x$ for all $x \in (0,\delta]$.
This shows that $d(x)=\min\{D^+M(m+x), -D^+(m-x)\}$ satisfies
\begin{equation} \label{dlb}
 d(x)\ge c \; x \quad \forall \; 0 < x \le \delta.
\end{equation}

Put $x_n:= x_n(\epsilon):= \epsilon n^{-\gamma}$. Then by Theorem \ref{expb}

\begin{equation} \label{ab}
 \mathbb{P}\big(n^\gamma \sup_{k \ge n}|\hat{m}_k-m| > \epsilon \big)= \mathbb{P}\big(\sup_{k \ge n}|\hat{m}_k-m| > x_n \big) \le 2 e^{-\frac{2}{A(x_n)^2} n d(x_n)^2} \quad \forall \; n \in \mathbb{N}.
\end{equation}
W.l.o.g. $\gamma>0$. Deduce from $x_n \downarrow 0$ that $A(x_n) \rightarrow \max\{b(0+)-a(0+),b(0-)-a(0-)\} =: A_0$, where w.l.o.g. $A_0$ is positive.
Indeed, in (\ref{b}) we may replace for instance the upper bound $b(x)$ by $b(x)+1$ resulting in $A_0 > 1$. Thus there exists an
integer $n_0$ such that $\frac{2}{A(x_n)^2} \ge \frac{1}{A_0^2}$ for all $n \ge n_0$. Moreover, $x_n \in (0,\delta]$ for all
$n \ge n_1$ with some $n_1 \in \mathbb{N}$, whence we can conclude from (\ref{dlb}) that
$\frac{2}{A(x_n)^2} n d(x_n)^2 \ge \frac{n c^2 x_n^2}{A_0^2}= L \epsilon^2 n ^{1-2\gamma}$ for every $n > n_2 =\max\{n_0,n_1\} \in \mathbb{N}$,
where $L=c^2/A_0^2$ is a positive constant. It follows with (\ref{ab}) that
$$
\sum_{n=1}^\infty n^{r-1} \mathbb{P}\big(n^\gamma \sup_{k \ge n}|\hat{m}_k-m| > \epsilon\big) \le \sum_{n=1}^{n_2} n^{r-1} + 2 \sum_{n > n_2} n^{r-1} e^{-L \epsilon^2 n^{1-2\gamma}} < \infty \quad \forall \; \epsilon >0
$$
and for all $r>0$, since $1-2 \gamma>0$.
\end{proof}

Since complete convergence $(r=1)$ implies almost sure convergence,
one can conclude from (\ref{rcompl}) that $\sup_{k \ge n}|\hat{m}_k-m|=o(n^{-\gamma})$ a.s. for all $\gamma<1/2$.
Our next result concerns the limit case $\gamma=1/2$, which involves an additional logarithmic factor. In fact the rate is $O(n^{-1/2} \sqrt{\log n})$. Actually, we can show a bit more:

\vspace{0.3cm}
\begin{corollary} \label{bigO} Under the assumptions of Corollary \ref{rcomplconv} it follows that there exists a deterministic real constant $K>0$ such that with probability one
$$
  \sup_{k \ge n}|\hat{m}_k-m| \le K n^{-1/2} \sqrt{\log n} \quad \text{ for eventually all } n \in \mathbb{N}.
$$
\end{corollary}

\begin{proof} Put $x_n := K n^{-1/2} \sqrt{\log n}$. By Theorem \ref{expb}
\begin{equation} \label{ineq}
\mathbb{P}\big(\sup_{k \ge n}|\hat{m}_k-m| > x_n \big) \le 2 e^{-\frac{2}{A(x_n)^2} n d(x_n)^2} \quad \forall \; n \in \mathbb{N}.
\end{equation}
Since $x_n \downarrow 0$ it follows as in the above proof that there exists some integer $n_0$ such that
$\frac{2}{A(x_n)^2} n d(x_n)^2 \ge A_0^{-2} n c^2 x_n^2= A_0^{-2} c^2 K^2 \log n$ for all $n \ge n_0$. Thus with the inequality (\ref{ineq}) we
obtain
$$
 \sum_{n \ge n_0} \mathbb{P}\big(\sup_{k \ge n}|\hat{m}_k-m| > x_n \big) \le 2 \sum_{n \ge n_0} n^{-c^2 K^2/A_0^2} < \infty,
$$
if $K > A_0/c$. Now an application of the 1. Borel-Cantelli Lemma gives the desired result.
\end{proof}

In the further course, we deal with \emph{$r-$quick convergence}, which goes back to Strassen (1967) \cite{Strassen}.
For real random variables $Z$ and $Z_n, n \in \mathbb{N}$, and $\epsilon >0$ one considers the random set of
indices $\mathbb{I}_\epsilon :=\{k \in \mathbb{N}: |Z_k-Z| > \epsilon\}$. Define $L_\epsilon := \sup \mathbb{I}_\epsilon$, if
$\mathbb{I}_\epsilon$ is non-empty and bounded.
Then $L_\epsilon \in \mathbb{I}_\epsilon$, whence $L_\epsilon+1$ is the first index from which the sequence $(Z_n)_{n \in \mathbb{N}}$ eventually stays in the region $[Z-\epsilon,Z+\epsilon]$.
It may happen that the sequence from the start lies in $[Z-\epsilon,Z+\epsilon]$, that is $\mathbb{I}_\epsilon = \emptyset$.
Then $L_\epsilon :=0$ describes this situation in a suitable way. Finally, in case that $\mathbb{I}_\epsilon$ is unbounded,
that is the sequence lies outside the interval $[Z-\epsilon,Z+\epsilon]$ infinitely often, then $L_\epsilon := \infty$ fits.
Now for $r>0$ the sequence $(Z_n)_{n \in \mathbb{N}}$ \emph{converges $\mathbb{P}-r-$quickley to} $Z$, if $\mathbb{E}[L_\epsilon^r] < \infty$ for all $\epsilon >0$.
This is denoted by $Z_n \rightarrow Z \; \mathbb{P}-r-$quickley. The following relation follows from the definition of $L_\epsilon$.
It holds for every positive real number $z$.
\begin{equation} \label{rel}
 L_\epsilon < z \quad \Longleftrightarrow \quad \sup_{k \ge [z]} |Z_k-Z| \le \epsilon,
\end{equation}
where $[z]:= \min\{k \in \mathbb{Z}: k \ge z \}$.

\vspace{0.3cm}
\begin{corollary} \label{quick}
If $m \in \mathbb{R}$ is the unique minimizer of $M$, then $\hat{m}_n \rightarrow m \; \mathbb{P}-r-$quickley for all $r>0$.
\end{corollary}

\begin{proof} Since
\begin{eqnarray*}
 \mathbb{P}(L_\epsilon^r \ge n) &=& \mathbb{P}(L_\epsilon \ge n^{1/r})\\
                                &=& \mathbb{P}( \sup_{k \ge [n^{1/r}]}|\hat{m}_k-m| > \epsilon) \quad \text{by } (\ref{rel})\\
                                &\le& 2 \exp\{-2 A(\epsilon)^{-2} [n^{1/r}] d(\epsilon)^2\} \quad \text{by Theorem } \ref{expb}\\
                                &\le& 2 \exp\{-2 A(\epsilon)^{-2} n^{1/r} d(\epsilon)^2\} \quad \text{by } [z] \ge z,
\end{eqnarray*}
it follows for each positive $\epsilon$ that $\sum_{n \ge 1} \mathbb{P}(L_\epsilon^r \ge n) < \infty$, because $A(\epsilon)^{-2}>0$ and $d(\epsilon) >0$ by uniqueness of $m$ as was shown in the proof of Theorem \ref{expb}. From the convergence of the series we can conclude
that $\mathbb{E}[L_\epsilon^r] < \infty$ for all $\epsilon>0$, which gives the desired result.
\end{proof}

Next notice that $n D^+M(m+x)^2= \delta_n(x)^2$, where the functions $\delta_n(x)=\sqrt{n} D^+M(m+x), n \in \mathbb{N},$ play a crucial role in the distributional convergence of the M-estimators $\hat m_n$. In fact in Theorem 3.4 of Ferger \cite{Ferger3} we show the following result:\\

Suppose $D^+M(m)=0=D^-M(m)$ and $\sigma^2:=\mathbb{E}[D^+h(X_1,m)^2] \in (0,\infty)$. Further assume that $(a_n)_{n \in \mathbb{N}}$ is a positive sequence converging to infinity such that
\begin{equation} \label{an}
\delta_n(x/a_n)=\sqrt{n} D^+M(m+x/a_n) \rightarrow \delta(x) \; \text{ for all } x \in \mathbb{R}.
\end{equation}

 Then $\delta:\mathbb{R} \rightarrow [-\infty,\infty]$ is increasing with $\delta \ge 0$ on $[0,\infty)$, $\delta \le 0$ on $(-\infty,0)$ and $\delta(0)=0$. Notice that $\delta$ can take the values $-\infty$ or $\infty$. In view of Smirnov's \cite{Smirnov} results for quantile esimators it is assumed that there are the following four types of limit functions. To introduce them put $I_\delta = \{t \in \mathbb{R}: -\infty < \delta(t) < \infty\}$.
We say that $\delta$ is of type 1, 2, 3 or 4 according as $I_\delta$ is equal to $[0,\infty), (-\infty,0], (-\infty,\infty)$ or $[-c_1,c_2]$, respectively, where $c_1,c_2 \ge 0$ and $\max\{c_1,c_2\}>0$. Moreover, the $\delta$'s of type 1-3 are supposed to be continuous and strictly increasing on $I_\delta$, and $\delta$ of type 4 is assumed to be equal to zero on $(-c_1,c_2)$ with finite one-sided limits at the end-points, which necessarily are equal to zero. In Theorem 3.4 \cite{Ferger3} we obtain:
\begin{equation} \label{dconv}
 \lim_{n \rightarrow \infty}\mathbb{P}(a_n(\hat{m}_n-m) \le x) =  \Phi_\sigma(\delta(x))=:H(x)  \quad \text{for all continuity points } x \text{ of } H,
\end{equation}
where $\Phi_\sigma$ is the distribution function of the centred normal distribution $N(0,\sigma^2)$ with variance $\sigma^2$.
Notice that the discontinuity points $x$ of $H$ are exactly the boundary points of $I_\delta$, i.e., $x=0$ for $\delta$ of type $1$ or $2$ and
$x \in \{-c_1,c_2\}$ for $\delta$ of type $4$, whereas for $\delta$ of type $3$ all $x \in \mathbb{R}$ are continuity points of $H$.
If $\delta$ is of type 1 or 3, then $H$ is a distribution function. In the other two cases, $H$ can be modified to a distribution function by
redefining $H(0):=1$ for type 2 and $H(-c_1):=1/2, H(c_2):=1$ for type 4, so that
the convergence (\ref{dconv}) still holds. This shows that $a_n(\hat{m}_n-m)$ converges in distribution to a random variable $V \sim H$.\\

Further conclusions from Theorem \ref{expb} are the following asymptotic tail bounds:\\

\begin{corollary} \label{atb} Under the assumptions of Theorem \ref{expb} suppose that $(a_n)_{n \in \mathbb{N}}$ satisfies condition (\ref{an}).
\begin{itemize}
\item[(1)]
If the right limits $\alpha_+:=a(0+) \in \mathbb{R}$ and $\beta_+:=b(0+) \in \mathbb{R}$ of $a$ and $b$ at $0$ exist, then
$$
 \limsup_{n \rightarrow \infty} \mathbb{P}(a_n \sup_{k \ge n} (\hat{m}_k-m) >x) \le \exp\{-2(\beta_+-\alpha_+)^{-2} \delta(x)^2\} \quad \forall \; x>0.
$$
\item[(2)]
If the left limits $\alpha_-:=a(0-) \in \mathbb{R}$ and $\beta_-:=b(0-)\in \mathbb{R}$ of $a$ and $b$ at zero exist, then
$$
 \limsup_{n \rightarrow \infty} \mathbb{P}(a_n \inf_{k \ge n} (\hat{m}_k-m) <x) \le \exp\{-2(\beta_--\alpha_-)^{-2} \delta(x)^2\} \quad \forall \; x<0.
$$
\item[(3)]
If the right and left limits of $a$ and $b$ at $0$ exist, then
$$
 \limsup_{n \rightarrow \infty} \mathbb{P}(a_n \sup_{k \ge n} |\hat{m}_k-m| >x) \le 2 \exp\{-2\tau^{-2} \Delta(x)^2\} \quad \forall \; x>0,
$$
where $\tau=\max\{\beta_+-\alpha_+,\beta_--\alpha_-\}$ and $\Delta(x)=\min\{\delta(x),-\delta(-x)\}$.
\end{itemize}
\end{corollary}

\begin{proof} We apply inequalities (\ref{i1}) and (\ref{i3}) with $x$ replaced by $x/a_n$ and inequality (\ref{i2}) with $x$ replaced by $-x/a_n$.
Then taking the limit $n \rightarrow \infty$ yields the result upon noticing (\ref{an}).
\end{proof}

\begin{remark} \label{conj} If $|\delta(x)| \rightarrow \infty, |x| \rightarrow \infty$, then it follows from the above corollary that
$$
 \lim_{x \rightarrow \infty} \limsup_{n \rightarrow \infty} \mathbb{P}(a_n \sup_{k \ge n} |\hat{m}_k-m| >x)=0.
$$
Thus the sequence $(a_n \sup_{k \ge n} |\hat{m}_k-m|)_{n \in \mathbb{N}}$ is tight and hence by Prokorov's Theorem it is
relatively compact. Since $|\sup S| \le \sup |S|$ and $|\inf S| \le \sup |S|$ for every non-empty $S \subseteq \mathbb{R}$, it follows
that the sequences \newline
$(a_n \sup_{k \ge n} (\hat{m}_k-m))_{n \in \mathbb{N}}$ and
$(a_n \inf_{k \ge n} (\hat{m}_k-m))_{n \in \mathbb{N}}$
are relatively compact as well. Given this, we conjecture that these three sequences actually converge in distribution.
\end{remark}

\vspace{0.4cm}
In the sequel we consider bivariate functions $h(x,t)=v(t-x)$ with convex $v:\mathbb{R} \rightarrow \mathbb{R}$.
As it becomes clear further below, it is convenient to construct the functions $v$ as follows:
Let $\varphi:\mathbb{R} \rightarrow \mathbb{R}$ be a nondecreasing and right-continuous function.
Then
\begin{equation} \label{vphi}
v(u):=v_\varphi(u):=\int_{0}^u \varphi(s) ds, \; u \in \mathbb{R},
\end{equation}
is convex, confer e.g. section 1.6 in Niculescu and Persson \cite{Niculescu}. Consequently, we can apply our results from section 3.
By Theorem 24.2 in Rockafellar \cite{Rockafellar}  $D^- v(u) = \varphi(u-)$ and $D^+v(u)=\varphi(u+)=\varphi(u)$ by right continuity. Thus the requirement $D^-v(0) \le 0 \le D^+v(0)$ is the same as
$\varphi(0-) \le 0 \le \varphi(0).$
Moreover, $v_\varphi$ is differentiable at a point $u \in \mathbb{R}$ if and only if $\varphi$ is continuous at $u$.
Typically, but not necessarily, the graph of $\varphi$ is roughly $S$-shaped, that is $\varphi$ is a sigmoid function, which is everywhere continuous, except with a possible jump at zero.
If actually $\varphi(u) < 0$ for all $u < 0$ and $\varphi(u)>0$ for all $u>0$ (in brief $\varphi <>0$), then $v_\varphi$ is strictly convex
and $m$ is unique. It satisfies the equation $\int_\mathbb{R} \varphi(m-s)Q(ds)=0$. Analogously, the estimator $\hat{m}_n$ is also unique and given as the solution $t=\hat{m}_n$ of $\int_\mathbb{R} \varphi(t-s) Q_n(ds)=0$.
Since $D^+h(X_i,m+x)=\varphi(m+x-X_i)$, the boundedness assumption (\ref{b}) is fulfilled, whenever $\varphi$ is bounded. Then a possible choice are the constant functions $a(x)=\inf_{u \in \mathbb{R}} \varphi(u)$ and $b(x)=\sup_{u \in \mathbb{R}} \varphi(u)$. If $\varphi$ is unbounded, but has an inverse $\varphi^{-1}$, then we need that the $X_i$ are bounded, say $A \le X_i \le B$ for
some constants $A<B$. In fact in that case (\ref{b}) is satisfied with $a(x)=\varphi(m+x-B)$ and $b(x)=\varphi(m+x-A)$ as one can easily verify.\\

Our first example is simple, but it yields a surprising result.\\

\begin{example} \label{id} (\textbf{Hoeffding}) Let $\varphi(u)=u$. Then $0=\int \varphi(m-s)Q(ds)=m-\mathbb{E}[X_1]$, whence $m=\mathbb{E}[X_1]=: \mu$.
Replacing $Q$ by $Q_n$ gives the estimator $\hat{m}_n=\bar{X}_n := 1/n \sum_{i=1}^n X_i$, the arithmetic mean of the data.
Moreover, $D^+M(m+x)=\int_\mathbb{R} \varphi(m+x-s) Q(ds)= m+x-\mu=x$. If $A \le X_i \le B$ for every $1 \le i \le n$, then
$b(x)-a(x)=B-A$. Thus for instance the exponential inequality (\ref{i3}) in Theorem \ref{expb} reduces to
\begin{equation} \label{supHoeffding}
 \mathbb{P}\Big(\sup_{k \ge n}|\hat{m}_k-m|>x\Big)=\mathbb{P}\Big(\sup_{k \ge n}|\bar{X}_k-\mu| > x\Big) \le 2 \exp\{-2 (B-A)^{-2}n x^2\}.
\end{equation}
Compare this with Hoeffding's inequality:
$$
  \mathbb{P}\Big(|\bar{X}_n-\mu| > x\Big) \le 2 \exp\{-2 (B-A)^{-2}n x^2\}.
$$
We see that in Hoeffding's inequality one can replace $|\bar{X}_n-\mu|$ by the much larger variable $\sup_{k \ge n}|\bar{X}_k-\mu| \ge |\bar{X}_n-\mu|$ while maintaining the same exponential bound. This is a substantial extension of Hoeffding's result.
\end{example}

\vspace{0.2cm}
When certain symmetries are present, the inequality (\ref{i3}) simplifies.\\

\begin{remark} \label{symcase} (\textbf{Symmetric case}) Assume that $Q$ is symmetric at $m \in \mathbb{R}$ and $\varphi<>0$ is an odd function in the sense that $\varphi(-u)=-\varphi(u)$ for every $u \neq 0$. Then $v_\varphi$ is even (and strictly convex). Therefore $m$ is the unique minimizer of
$M$. Moreover,
\begin{eqnarray*}
D^+M(m-x)&=&\mathbb{E}[\varphi(m-x-X_1)]=-\mathbb{E}[\varphi(-m+x+X_1)]=-\mathbb{E}[\varphi(m+x-X_1)]\\
         &=&-D^+(m+x) \quad \forall \; x \in \mathbb{R}.
\end{eqnarray*}
If in addition $\varphi(0)=0$, i.e., $\varphi(-u)=-\varphi(u)$ for all $u \in \mathbb{R}$, then $a(-x)=-b(x)$ and $b(-x)=-a(x)$ for all $x>0$, whence
$A(x)=b(x)-a(x)$ and $d(x)=D^+M(m+x)$.
In summary we get
$$
\mathbb{P}(\sup_{k \ge n} |\hat{m}_k-m| >x) \le 2 \exp\{-2 (b(x)-a(x))^{-2} n D^+M(m+x)^2\} \quad \forall \; x>0.
$$
If $\varphi$ is bounded, then the constant $2 (b(x)-a(x))^{-2}$ can be replaced by $\frac{1}{2}||\varphi||^{-2}$, where $||\varphi||$ is the sup-norm of $\varphi$.
\end{remark}

\vspace{0.5cm}
Recall the assumption (\ref{DpMlb}), which was pivotal for the speed of $r-$complete convergence. Suppose $\varphi$ is differentiable on $\mathbb{R}$ with bounded derivative, i.e., $||\varphi^\prime||< \infty$. If $D^+M(m)=0$, then
\begin{eqnarray}
 D^+M(m+x)&=&D^+M(m+x)-D^+M(m)=\int_\mathbb{R} \varphi(m+x-s)-\varphi(m-s) Q(ds) \nonumber\\
          &=&\int_\mathbb{R}\frac{\varphi(m+x-s)-\varphi(m-s)}{x} Q(ds) \;x. \label{lbphi}
\end{eqnarray}
Since the integrand in the last integral (\ref{lbphi}) is dominated by $||\varphi^\prime||$, it follows from the Dominated Convergence Theorem that
\begin{equation} \label{intphi}
\int_\mathbb{R}\frac{\varphi(m+x-s)-\varphi(m-s)}{x} Q(ds) \rightarrow \int_\mathbb{R} \varphi^\prime(m-s) Q(ds), \;x \rightarrow 0.
\end{equation}
Assume that $\varphi^\prime$ is positive, i.e., $\varphi^\prime(u) >0$ for all $u \in \mathbb{R}$. Then the integral on the right side of (\ref{intphi}) is positive.
Indeed, the integrand $\varphi^\prime(m-s) \ge 0$ for all $s \in \mathbb{R}$, because $\varphi$ is monotone increasing. Suppose the integral is equal to zero. Then $\varphi^\prime(m-s)=0$ for $Q-$almost all $s \in \mathbb{R}$. In particular, there exists some $s \in \mathbb{R}$ such that $\varphi^\prime(m-s)=0$ in contradiction to positiveness of $\varphi^\prime$. Infer from (\ref{intphi}) that
\begin{equation} \label{halb}
 \int_\mathbb{R}\frac{\varphi(m+x-s)-\varphi(m-s)}{x} Q(ds) \ge \frac{1}{2}\int_\mathbb{R} \varphi^\prime(m-s) Q(ds) \quad \forall \; x \in [-\delta,\delta]
\end{equation}
for some $\delta$ sufficiently small. Combining (\ref{lbphi}) with (\ref{halb}) leads to the following\\

\begin{lemma} \label{pivot} Assume that $\varphi$ in (\ref{vphi}) has a positive and bounded derivative $\varphi^\prime$ and that $D^+M(m)=0$.
Then  there exists a real $\delta>0$ such that
\begin{equation} \label{DpMlbphi}
 |D^+M(m+x)| \ge c \; |x| \quad \forall \; x \in [-\delta,\delta],
\end{equation}
where $c=\frac{1}{2}\int_\mathbb{R} \varphi^\prime(m-s) Q(ds)$ is positive. (In fact the constant $\frac{1}{2}$ can be replaced by any
constant strictly smaller than $1$.)
\end{lemma}

\vspace{0.4cm}
Finally, we would like to determine the crucial function
\begin{equation} \label{Qint}
K(x):=D^+M(m+x)=\int_\mathbb{R}\varphi(m+x-s) Q(ds), \; x \in \mathbb{R}.
\end{equation}
We have already seen two examples. The $\alpha-$quantile arises from $\varphi(u)= 1_{\{u \ge 0\}}-\alpha$,
which yields $D^+M(m+x)=F(m+x)-\alpha$. And $\varphi(u)=u$ induces $D^+M(m+x)=x$ as shown in Example \ref{id}.\\

The following observation is very useful, because it makes the integral in (\ref{Qint}) in some situations analytically tractable.
Namely, if $Q$ has a Lebesgue-density $f$, then $K(x)= \varphi \star f(m+x)$,
where $\star$ denotes the convolution-operator.
Thus by the Differentiation-Lemma (which can be applied for instance if $||\varphi^\prime||< \infty$) it follows that
\begin{equation} \label{diffeq}
 K^\prime(x)= \varphi^\prime \star f(m+x).
\end{equation}
Consequently, $K(x)=\int \varphi^\prime \star f(m+x) dx + C$ is the primitive, where the integration constant $C$ is uniquely determined by the constrained $K(0)=0$.\\

\begin{example} (\textbf{Normal})\label{normal} Assume $F=\Phi_{\mu,\sigma^2}$ is the distribution function of the normal law $N(\mu,\sigma^2)$ with pertaining density
$$
\phi_{\mu,\sigma^2}(s)=\frac{1}{\sqrt{2 \pi \sigma^2}} \exp\{-\frac{(s-\mu)^2}{2 \sigma^2}\}.
$$
Let $\varphi(u)=\Phi(u)-1/2$ with $\Phi:=\Phi_{0,1}$. Then $m=\mu$ according to Remark \ref{symcase} and
$$
K(x)=\int_\mathbb{R}(\Phi_{0,1}(\mu+x-s)-1/2) N(\mu,\sigma^2)(ds)= \int_\mathbb{R}(\Phi_{0,1}(\mu+x-s)-1/2) \phi_{\mu,\sigma^2}(s)ds.
$$
By the Differentiation-Lemma
$$
K^\prime(x)=\int_\mathbb{R}\phi_{0,1}(\mu+x-s) \phi_{\mu,\sigma^2}(s)ds =\phi_{0,1} \star \phi_{\mu,\sigma^2}(\mu+x)=\phi_{\mu,1+\sigma^2}(\mu+x)=\phi_{0,1+\sigma^2}(x),
$$
where the second last equality is the additive property of the normal distribution. Integration yields
$$
 D^+M(m+x)= \Phi\Big(\frac{x}{\sqrt{1+\sigma^2}}\Big)-1/2.
$$
\end{example}

\begin{example} (\textbf{Cauchy}) Let $F=F_a$ be the Cauchy-distribution function with location parameter $a \in \mathbb{R}$ and let $\varphi=F_0-1/2$. Then
by symmetry $m=a$. Using the additive property of the Cauchy-distribution, the same arguments as in the example above lead to
$$
 D^+M(m+x)=\frac{1}{\pi} \arctan(x).
$$
\end{example}

\begin{example} (\textbf{Uniform}) \label{uniform} For reals $a<b$ let $F$ be the distribution function of $U(a,b)$, the uniform distribution on $[a,b]$.
Consider $\varphi=\Phi-1/2$, so that by symmetry $m=(a+b)/2$. Furthermore, if $\phi:=\phi_{0,1}$, then
$$
 K^\prime(x)= \int_\mathbb{R} \phi(m+x-s) U(a,b)(ds)=\frac{1}{b-a} \int_a^b \phi(m+x-s) ds = \frac{1}{b-a} \int_{m+x-b}^{m+x-a} \phi(u)du.
$$
Consequently, $K^\prime(x)=\frac{1}{b-a}\{\Phi(m+x-a)-\Phi(m+x-b)\} = \frac{1}{b-a}\{\Phi(x+(b-a)/2)-\Phi(x-(b-a)/2)\}$.
Now, the primitive of $\Phi(s)$ is equal to $H(s):=s \Phi(s)+\phi(s)$ as one can easily verify by differentiation.
Thus
$$
 K(x)=\frac{1}{b-a}\{H(x+(b-a)/2)-H(x-(b-a)/2)\}-1/2,
$$
where the constant $1/2$ arises from the constraint $K(0)=0$.
\end{example}

\begin{example} As already seen $\varphi(u)=1_{\{u \ge 0\}}-1/2$ corresponds to the median ($\alpha=1/2$). This is a simple step function
with value $-1/2$ on $(-\infty,0)$ and $1/2$ on $[0,\infty)$. We change this discontinuous function to a continuous function.
For $c>0$ let $\varphi_c=\varphi$ outside the interval $[-c,c]$ and on that interval let $\varphi_c$ be the straight line connecting the points
$(-c,-1/2)$ and $(c,1/2)$ i.e., $\varphi_c(u)= \frac{1}{2c} u, u \in [-c,c]$. Assume $Q$ is symmetric at $0$. From Remark \ref{symcase} we know that $m=0$. Thus by (\ref{Qint})
\begin{eqnarray*}
& &D^+M(x) = \int_\mathbb{R} \varphi_c(x-s) Q(ds)\\
         &=& \int_{(-\infty,x-c]} 1/2 Q(ds)-\int_{(x+c,\infty)} 1/2 Q(ds)+\frac{1}{2c}\int_{(x-c,x+c]}x-s Q(ds)\\
         &=& 1/2 F(x-c)-1/2(1-F(x+c))+\frac{1}{2c}x(F(x+c)-F(x-c))-\frac{1}{2c}\int_{x-c}^{x+c} s F(ds).
\end{eqnarray*}
Integration by parts yields
$$
 \int_{x-c}^{x+c} s F(ds)= (x+c)F(x+c)-(x-c)F(x-c)-\int_{x-c}^{x+c} F(s) ds,
$$
so that after some simple algebra we arrive at
\begin{equation} \label{int}
 D^+M(x)= \frac{1}{2c} \int_{x-c}^{x+c} F(s) ds -\frac{1}{2}.
\end{equation}

If for instance $F=\Phi$, then
$$
\int_{x-c}^{x+c} F(s) ds=\int_{x-c}^{x+c} \Phi(s) ds =(x+c)\Phi(x+c)-(x-c)\Phi(x-c)+\phi(x+c)-\phi(x-c).
$$
This is so, because the primitive of $\Phi(s)$ is equal to $s \Phi(s)+\phi(s)$ as we know from
from Example \ref{uniform}.

Observe that $\frac{1}{2c} \int_{x-c}^{x+c} F(s) ds \rightarrow F(x)$ as $c \rightarrow 0$, which leads us back to $D^+M(x)=F(x)-1/2$ pertaining to the median.
\end{example}

\vspace{0.2cm}
Our next example combines the median and the mean. It goes back to Huber \cite{Huber}, confer also Serfling \cite{Serfling}, p.247.\\

\begin{example} (\textbf{Huber}) Assume that $Q$ is symmetric at $0$. For some fixed $c>0$ put
$$
 \varphi(u):= \left\{ \begin{array}{l@{\quad,\quad}l}
                 -c & u <c \\ u & |u|\le c \\
                 c  & u >c
               \end{array} \right.
$$
Again according to Remark \ref{symcase} we have that $m=0$ and similarly as in the example before one obtains:
$$
 D^+M(x)=\int_{x-c}^{x+c}F(s) ds -c.
$$
\end{example}

\begin{example}
Let $\varphi(u):=u/\sqrt{1+u^2}, u \in \mathbb{R},$ and $Q$ be the uniform distribution on $[-1,1]$. Then $m=0$ and
$$
 D^+M(x)= \int_\mathbb{R} \varphi(x-s) Q(ds)= \frac{1}{2} \int_{-1}^1 \varphi(x-s)ds=\frac{1}{2} \int_{x-1}^{x+1} u/\sqrt{1+u^2} du
$$
by substitution $u=x-s$. Now check that $\sqrt{u^2+1}$ is the primitive of $u/\sqrt{1+u^2}$, whence we obtain:
$$D^+M(x)=\frac{1}{2}\Big(\sqrt{(x+1)^2+1}-\sqrt{(x-1)^2+1}\Big)$$.
\end{example}

\section{Polynomial bounds under moment-conditions}
Assume that instead of the boundedness condition (\ref{b}) we only have that $D^+h(X_1,\cdot)$ possesses $s$-th absolute moments in an arbitrarily small
neighborhood of $m$. More precisely let $s \ge 1$ and suppose that there exists some $x_0>0$ such that:
\begin{equation} \label{rm}
 R_s(x):=\mathbb{E}[|D^+h(X_1,m+x)|^s] < \infty \quad \text{for all } x \in [-x_0,x_0].
\end{equation}

\begin{theorem} \label{poly} If (\ref{rm}) holds with $s \ge 2$ and $m$ is unique, then there exist positive real constants $B_s$ such that
for all $x \in (0,x_0]$ the following inequalities hold:
\begin{align}
&\mathbb{P}(\sup_{k \ge n} (\hat{m}_k-m) >x) \le 2^s B_s R_s(x) \big(D^+M(m+x)\big)^{-s}  n^{-s/2}, \label{in1}\\
&\mathbb{P}(\inf_{k \ge n} (\hat{m}_k-m) < -x) \le  2^s B_s R_s(x) \big(-D^+M(m-x)\big)^{-s}  n^{-s/2}, \label{in2}\\
&\mathbb{P}(\sup_{k \ge n} |\hat{m}_k-m| >x) \le 2^{s+1} B_s \mu_s(x) d(x)^{-s}  n^{-s/2}, \label{in3}
\end{align}
where
$$
d(x)=\min\{D^+M(m+x),-D^+M(m-x)\}>0, \;\mu_s(x)=\max\{R_s(x),R_s(-x)\}<\infty.
$$
\end{theorem}

\begin{proof} We begin with the proof of (\ref{in1}). Recall that $\eta = D^+M(m+x)$ is positive by Theorem \ref{expb}. (The proof of this fact does not use the boundedness assumption.)
From (\ref{mart}) it follows that
\begin{eqnarray}
\mathbb{P}(\bigcup_{n \le k \le l} \{\hat{m}_k-m >x\}) &=& \mathbb{P}(\bigcup_{1 \le j \le l-n+1} \{T_{l-j+1} > \eta\})=\mathbb{P}(\bigcup_{1 \le j \le l-n+1} \{|T_{l-j+1}|^s > \eta^s\}) \nonumber\\
&=& \mathbb{P}(\max_{1 \le j \le l-n+1} |T_{l-j+1}|^s > \eta^s)\nonumber\\
&\le & \eta^{-s} \mathbb{E}[|T_n|^s] \qquad \qquad \text{ by Doob's inequality, because } \eta \text{ is positive} \nonumber\\
&=& \eta^{-s} n^{-s} \mathbb{E}[|\sum_{i=1}^n \xi_i(x)|^s]. \label{Doob}
\end{eqnarray}
Since
\begin{eqnarray*}
 |\xi_i(x)|^s&=&|D^+h(X_i,m+x)-\mathbb{E}[D^+h(X_i,m+x)]|^s \le 2^{s-1} \{|D^+h(X_i,m+x)|^s+|\mathbb{E}[D^+h(X_i,m+x)]|^s\} \nonumber\\
 &\le& 2^{s-1}\{|D^+h(X_i,m+x)|^s+\mathbb{E}[|D^+h(X_i,m+x)|^s]\} \qquad \text{ by Jensen's inequality}, \label{Jensen}
\end{eqnarray*}
it follows that
\begin{equation}
 \mathbb{E}[|\xi_i(x)|^s] \le 2^s \mathbb{E}[|D^+h(X_i,m+x)|^s]=2^s R_s(x) < \infty \quad \text{ by } (\ref{rm}).
\end{equation}
For $s>2$ a combination of H\"{o}lder's inequality with the Marcinkiewicz-Zygmund inequality yields, confer, e.g., (1.5) in Ferger (2014) \cite{Ferger1}:
\begin{equation} \label{MZH}
 \mathbb{E}[|\sum_{i=1}^n \xi_i(x)|^s] \le B_s n^{s/2} \mathbb{E}[|\xi_1(x)|^s].
\end{equation}
For $s=2$ the above inequality holds (with $B_2=1$) by Bienaym\'{e}'s equality. Thus (\ref{Doob}) - (\ref{MZH}) together with (\ref{obs1}) result in inequality (\ref{in1}). For the derivation of (\ref{in2}) we can start with (\ref{inf1}) and (\ref{inf2}), which make us to proceed as above, because $\eta:=-D^+M(m-x)$ is positive by Theorem \ref{expb}. Finally, the third inequality (\ref{in3}) follows from (\ref{supinf}), (\ref{in1}) and (\ref{in2}).
\end{proof}

In the following we derive the counterparts of Corollaries \ref{rcomplconv}, \ref{quick} and \ref{atb}. As to the first one we use a uniform version of the moment-condition (\ref{rm}):
\begin{equation} \label{grm}
 L:=L_s:=\sup_{|x| \le x_0} \mathbb{E}[|D^+h(X_1,m+x)|^s] < \infty.
\end{equation}

\begin{corollary} Suppose (\ref{DpMlb}) is fulfilled  and (\ref{grm}) holds with $s > 2$. Then
\begin{equation} \label{gamma}
 \sup_{k \ge n}|\hat{m}_k-m|=o(n^{-\gamma}) \;\; \mathbb{P} \text{-r completely for all } \gamma < \frac{1}{2} \; \text{ and for all } r< (1/2-\gamma)s.
\end{equation}
\end{corollary}

\begin{proof} Put $x_n:=\epsilon n^{-\gamma}$. From the proof of Corollary \ref{rcomplconv} we know that $d(x) \ge c x$ for all $x \in (0,\delta]$, confer (\ref{dlb}). If (\ref{gamma}) is valid for some $\gamma$, then a fortiori for each smaller one. Thus we may assume that $\gamma >0$. Then there exists an integer $n_0$ such that $0<x_n \le \min\{\delta, x_0\}$. Infer from Theorem \ref{poly} that
\begin{eqnarray*}
 \mathbb{P}(n^\gamma \sup_{k \ge n}|\hat{m}_k-m|>\epsilon) &=& \mathbb{P}(\sup_{k \ge n}|\hat{m}_k-m|>x_n)
    \le 2^{s+1} B_s d(x_n)^{-s} \mu_s(x_n) n^{-s/2}\\
     &\le& 2^{s+1} B_s c^{-s} x_n^{-s} L n^{-s/2} = D \epsilon^{-s} n^{s(\gamma-1/2)},
\end{eqnarray*}
with constant $D= 2^{s+1} B_s c^{-s} L$. Thus we arrive at
$$
 \sum_{n \ge n_0} n^{r-1} \mathbb{P}(n^\gamma \sup_{k \ge n}|\hat{m}_k-m|>\epsilon) \le D \epsilon^{-s} \sum_{n \ge n_0} n^{r-1+s(\gamma-1/2)} < \infty,
$$
because the exponent $r-1+s(\gamma-1/2) < -1$. This completes our proof.
\end{proof}

Of course $\gamma<0$ or $r<1$ are uninteresting. The extreme constellation $\gamma=0$ and $r=1$ gives complete and in particularly a.s. convergence. Note that our assumptions cover this case, since $(1/2-\gamma)s=s/2>r=1$. To obtain the convergence rate $o(n^{-\gamma})$ with positive $\gamma$ the order $s$ in the uniform moment-condition (\ref{grm}) must be sufficiently high, namely $s> 1/(1/2-\gamma)$.\\

\begin{corollary} If $m \in \mathbb{R}$ is the unique minimizer of $M$ and (\ref{rm}) holds for some $s>2$, then $\hat{m}_n \rightarrow m \; \mathbb{P}$-r quickley for all $r \in (0,s/2)$.
\end{corollary}

\begin{proof} Using the equivalence (\ref{rel}) it follows from Theorem \ref{poly} that
$$
\mathbb{P}(L_\epsilon^r \ge n)=\mathbb{P}(\sup_{k \ge [n^{1/r}]}|\hat{m}_k-m|>\epsilon) \le 2^{s+1} B_s d(\epsilon)^{-s} \mu_s(\epsilon) n^{-\frac{s}{2r}},
$$
whence $\sum_{n \ge 1}\mathbb{P}(L_\epsilon^r \ge n) < \infty$ for all $\epsilon \in (0,x_0]$. If $\epsilon > x_0$, then
$L_\epsilon \le L_{x_0}$ and so the series is finite actually for all $\epsilon >0$, which in turn ensures that $\mathbb{E}[L_\epsilon^r]<\infty$
for every $\epsilon>0$ as desired.
\end{proof}

Finally, the counterpart of Corollary \ref{atb} reads as follows:\\

\begin{corollary} Assume that (\ref{grm}) holds with $s \ge 2$ and that $m$ is unique. If the sequence $(a_n)$ satisfies (\ref{an}), then there exists a positive constant $K_s$ such that:
\begin{align}
& \limsup_{n \rightarrow \infty} \mathbb{P}(a_n \sup_{k \ge n} (\hat{m}_k-m) >x) \le K_s \delta(x)^{-s} \quad \forall \; x>0.\\
& \limsup_{n \rightarrow \infty} \mathbb{P}(a_n \inf_{k \ge n} (\hat{m}_k-m) <x) \le K_s (-\delta(x))^{-s} \quad \forall \; x<0.\\
& \limsup_{n \rightarrow \infty} \mathbb{P}(a_n \sup_{k \ge n} |\hat{m}_k-m| >x) \le 2 K_s \Delta(x)^2 \quad \forall \; x>0,
\end{align}
where $\Delta(x)=\min\{\delta(x),-\delta(-x)\}$.
\end{corollary}

\begin{proof} The proof follows from Theorem \ref{poly}, where $K_s = 2^s B_s L_s$.
\end{proof}

As a consequence, the conjecture we made in Remark \ref{conj} is also appropriate here.\\

Our last result yields Niemiro's \cite{Niemiro} Theorem 2, however only for dimension $d=1$. We state it here, because in comparison to Niemiro our proof is really short.\\

\begin{proposition} \label{Thm2} Suppose $m \in \mathbb{R}$ is the unique minimizer of $M$ and (\ref{rm}) holds for some $s>1$. Then, for every $x>0$,
$$
 \mathbb{P}(\sup_{k \ge n} |\hat{m}_k-m| >x)=o(n^{-s+1}), \quad n \rightarrow \infty.
$$
\end{proposition}

\begin{proof} Observe that
\begin{eqnarray*}
 \mathbb{P}(\sup_{k \ge n} (\hat{m}_k-m) >x)&=&\mathbb{P}(\bigcup_{k \ge n}\{\hat{m}_k-m>x\})=\mathbb{P}(\bigcup_{k \ge n}\{T_k > \eta\})\\
                                            &\le& \mathbb{P}(\sup_{k \ge n} |\frac{1}{k} \sum_{i=1}^k \xi_i|> \eta)=o(n^{-s+1}),
\end{eqnarray*}
where the second equality holds by (\ref{rep}) and the last equality follows from Theorem 28 in Petrov \cite{Petrov}  taking into account the equivalence of (4.1) and (4.4) in Theorem 27 in \cite{Petrov}. Analogously, one obtains that $$\mathbb{P}(\inf_{k \ge n} (\hat{m}_k-m) <-x) = o(n^{-s+1}),$$ whence the assertion follows from (\ref{supinf}).
\end{proof}

Let $\hat{m}_n$ be the \textbf{largest minimizing point} of $M_n$. Then all results above remain valid with $D^+M$ and $D^+h$ replaced by
$D^-M$ and $D^-h$, respectively. This is the case due to the second part of Theorem 1 of Ferger \cite{Ferger0}.



\end{document}